\documentclass[usletter,10 pt,conference,times]{ieeeconf}  

\usepackage[utf8]{inputenc} 
\usepackage[T1]{fontenc}      
\usepackage{amsmath} 
\usepackage{amsfonts}  
\usepackage{graphicx} 
\usepackage{ulem}
\usepackage{enumerate}
\usepackage{subfigure}
\usepackage{algorithm}
\usepackage{algpseudocode}
\usepackage{xfrac}
\usepackage{array}

\IEEEoverridecommandlockouts                              

\newtheorem{rk}{Remark}

\newtheorem{ass}{Assumption}
\newtheorem{theo}{Theorem}
\newtheorem{lem}{Lemma}
\newtheorem{definition}{Definition}

\newcommand{\Bcal}{\mathcal{B}}

\newcommand{\Ncal}{\mathcal{N}}

\newcommand{\Scal}{\mathcal{S}}

\newcommand{\Zcal}{\mathcal{Z}}

\newcommand{\Rset}{\mathbb{R}}

\newcommand{\Nset}{\mathbb{N}}

\newcommand{\Trans}{\scriptscriptstyle\top}
\DeclareMathOperator*{\argm}{argmin}

\DeclareMathOperator*{\minimise}{minimise}

\DeclareMathOperator*{\infim}{inf}

\newcommand{\ri}{\operatorname{ri}}
\newcommand{\cl}{\operatorname{cl}}

\title{\LARGE \bf
An Augmented Lagrangian Coordination-Decomposition Algorithm for Solving Distributed Non-Convex Programs
}

\author{Jean-Hubert Hours and Colin N. Jones
\thanks{ Jean-Hubert Hours and Colin N. Jones are with the Laboratoire d'Automatique,~\'Ecole Polytechnique F\'ed\'erale de Lausanne,~Switzerland. 
        {\tt\small \{jean-hubert.hours, colin.jones\}@epfl.ch}}%
}

\begin{document}
\maketitle
\begin{abstract}
A novel augmented Lagrangian method for solving non-convex programs with nonlinear cost and constraint couplings in a distributed framework is presented. The proposed decomposition algorithm is made of two layers: The outer level is a standard multiplier method with penalty on the nonlinear equality constraints, while the inner level consists of a block-coordinate descent (BCD) scheme. Based on standard results on multiplier methods and recent results on proximal regularised BCD techniques, it is proven that the method converges to a KKT point of the non-convex nonlinear program under a semi-algebraicity assumption.~Efficacy of the algorithm is demonstrated on a numerical example. 
\end{abstract}
\section{Introduction}
\label{sec:intro}
When dealing with a large-scale system, such as a power grid for instance, sub-systems are coupled in a complex manner through their dynamics, implying that local control actions may have a major impact throughout the whole network. Implementing NMPC controllers for such systems may result in large-scale non-separable non-convex nonlinear programs (NLP). Solving such programs online in a centralised fashion is a challenging task from a computational point of view. Moreover, autonomy of the agents is likely to be hampered by such a centralised procedure. Therefore, distributed MPC is currently raising much interest, as an advanced control strategy for decentralising computations, thus reducing the computational burden and enabling autonomy of agents, which then only need to solve a local NLP \cite{neco2011}. Potential applications of distributed NMPC abound, from networked systems' control (interconnected power plants) to collaborative control (formation flying).\\
Decomposing a convex NLP is generally performed by applying Lagrangian decomposition methods \cite{bert1997}. A critical requirement for this strategy is strong duality, which is guaranteed by Slater's condition in the convex case. However, such an assumption rarely holds in a non-convex setting. Nevertheless, in an augmented Lagrangian setting, it has been shown that the duality gap can be locally removed by using appropriate penalty functions \cite{rock1974}. Moreover, augmented Lagrangian methods turn out to be much more efficient than standard penalty approaches from a computational point of view, as the risk of  ill-conditioning associated with large penalty parameters is reduced by the fast convergence of the dual iteration. However, the resulting quadratic penalty term is not separable even if the coupling constraints are. Several approaches have been explored to remedy this issue \cite{ham2011}. For instance, in \cite{bert1979,tani1985}, a local convexification procedure by means of proximal regularisation is analysed. A more recent approach to the decomposition of non-convex programs is the sequential programming scheme presented in \cite{neco2009}, which consists in iterative convex approximations and decompositions. Our decomposition strategy is not specifically targeted at separable NLPs, as objective and constraints couplings can be addressed by the proposed BCD scheme. It also differs from \cite{neco2009} in that only local convexification is required to guarantee convergence.\\
\looseness-1Our approach is an augmented Lagrangian method with partial constraint penalisation on the nonlinear equality constraints. It is a two-level optimisation scheme, whose outer layer is a loop on the Lagrange multipliers estimates associated with the nonlinear equality constraints. The inner level consists in approximately solving the primal problem to a critical point. This is performed via an inexact proximal regularised BCD scheme. By means of recent results on non-convex proximal regularised BCD methods \cite{att2013}, we prove convergence of the inner loop to a critical point of the non-convex inner problem under the assumption that the NLP is semi-algebraic. This is the main novelty of our approach, since, until recently, BCD type methods were thought to only have convergence guarantees under the very restrictive assumption that the NLP is convex. Some studies have been conducted in the non-convex case, yet the convergence results are quite weak compared to \cite{att2013}. The inner loop is the key step of the algorithm, which enables distributed computations. In general, as the augmented Lagrangian is not separable, inner iterations cannot be fully parallelised, but some degree of parallelism can be achieved if the interconnection graph is sparse \cite{bert1997}. Finally, convergence of our two-level optimisation technique to a KKT point of the original NLP is proven under standard assumptions. Eventually, the whole algorithm consists in iteratively solving decoupled Quadratic Programs (QP) and updating multipliers estimates in a coordination step.\\
In Section \ref{sec:pb_form}, the general framework of our algorithm is presented. In Sections \ref{sec:out_loop} and \ref{sec:inn_loop}, the algorithm is described and its convergence analysed. Finally, a numerical example is presented in Section \ref{sec:num_ex}.
\section{Background}
\label{sec:background}
\begin{definition}[Normal cone to a convex set]
Let ${\Omega}$ be a convex set in ${{\Rset}^{n}}$ and ${\bar{x}\in{\Omega}}$. The normal cone to ${\Omega}$ at $\bar{x}$ is the set
\begin{align}
{\Ncal}_{\Omega}(\bar{x}):=\left\{v\in\Rset^n~\Big|~\forall{}x\in\Omega,~v^{\Trans}(x-\bar{x})\leq{}0\right\}.
\end{align}
\end{definition}
\begin{lem}[Descent Lemma]
Let ${f\in{C}^2({\Omega},{\Rset})}$ (twice continuously differentiable in ${\Omega}$), where ${\Omega}$ is a convex compact set in ${\Rset}^{n}$. Let ${x,y\in\Omega}$.
\begin{align}
f(y)\leq f(x)+\nabla{}f(x)^{\Trans}(y-x)+\frac{\left\|\nabla^2f\right\|_{\infty}^{\Omega}}{2}\left\|y-x\right\|_2^2
\end{align}
where
\begin{align}
\left\|{\nabla}^2f\right\|_{\infty}^{\Omega}=\max\left\{\left\|\nabla^2f(x)\right\|_{2}~\Big|~x\in\Omega\right\}.
\end{align}
\looseness-1Given ${M\in\Rset^{n\times n}}$, $\left\|M\right\|_{2}$ denotes the induced matrix $2$-norm.
\end{lem}
\begin{definition}[Regular point]
Let $h\in C^1(\Rset^n,\Rset^m)$. A point ${x\in\Rset^n}$ such that ${h(x)=0}$ is called \textit{regular} if the gradients ${\nabla}h_1(x),\ldots,{\nabla}h_m(x)$ are linearly independent.
\end{definition}
\begin{definition}[Critical point]
\label{def:crit_pt}
Let $f$ be a proper lower semicontinuous function. A necessary condition for $x^\ast$ to be a minimiser of $f$ is that 
\begin{align}
\label{eq:crit_pt_def}
0\in{\partial}f(x^\ast),
\end{align}
where ${\partial}f(x^\ast)$ is the sub-differential of $f$ at $x^\ast$ \cite{rock1998}. Points satisfying \eqref{eq:crit_pt_def} are called \textit{critical points}. 
\end{definition}
The indicator function of a closed subset ${S}$ of ${\Rset^n}$ is denoted ${\delta_S}$ and is defined as
\begin{align}
\delta_S(x) = \begin{cases} 0 &\mbox{if } x \in S \\
+\infty & \mbox{if } x \notin S.\end{cases} 
\end{align}
\begin{lem}[Sub-differential of indicator function \cite{rock1998}]
Given a convex set $C$, for all $x\in C$,
\begin{align}
\partial\delta_C(x)=\Ncal_C(x). 
\end{align}
\end{lem}
\section{Problem formulation}
\label{sec:pb_form}
We consider NLPs of the following form:
\begin{align}
\label{eq:nlpdef}
&\minimise_{z_1,\ldots,z_N}~\sum_{i=1}^NJ_i(z_i)+Q(z_1,\ldots,z_N)\nonumber\\
&\text{s.t.}\nonumber\\
&~~~~~~F_i(z_i)=0,~i\in\left\{1,\ldots,N\right\},\nonumber\\
&~~~~~~G(z_1,\ldots,z_N)=0,\nonumber\\
&~~~~~~z_i\in\Zcal_i,~i\in\left\{1,\ldots,N\right\},
\end{align}
where \({z_i\!\in\!\Rset^{n_i}}\) for \({i\!\in\!\left\{1,\ldots,N\right\}}\) and \(\Zcal_i\) are polytopes in \(\Rset^{n_i}\). The term ${Q(z_1,\ldots,z_N)}$ is a \textit{cost coupling} term and ${G(z_1,\ldots,z_N)\in\Rset^p}$ a \textit{constraint coupling} term. They model the different kinds of systems' interactions, which may appear in a distributed NMPC context.\\ 
For the remainder of the paper, we also define the vector ${z\!:=\!\big(z_1^{\Trans},\ldots,z_N^{\Trans}\big)^{\Trans}\!\in\!\Rset^n}$, with dimension ${n\!:=\!\sum_{i=1}^Nn_i}$. For \({i\!\in\!\left\{1,\ldots,N\right\}}\), as ${\Zcal_i}$ is polytopic, there exists ${A_i\!\in\!\Rset^{q_i\times n_i}}$ and ${b_i\!\in\!\Rset^{q_i}}$ such that \({\Zcal_i\!:=\!\left\{x\!\in\!\Rset^{n_i}\big|A_ix\leq{}b_i\right\}}\). One can thus introduce the polytope \({\Zcal\subset\Rset^n}\) by \({\Zcal\!:=\!\left\{x\in\Rset^n\big|Ax\leq{}b\right\}}\), where \({A\!\in\!\Rset^{q\times n}}\), \({b\!\in\!\Rset^q}\) and \({q\!:=\!\sum_{i=1}^Nq_i}\). By posing ${m:=\sum_{i=1}^{N}m_i}$, we also define 
\begin{align}
H(z):=\big(F_1(z_1)^{\Trans},\ldots,F_N(z_N)^{\Trans},G(z_1,\ldots,z_N)^{\Trans}\big)^{\Trans}\in\Rset^r,
\end{align}
where ${r:=m+p}$. The distributed objective is 
\begin{align}
J(z):=\sum_{i=1}^NJ_i(z_i)+Q(z_1,\ldots,z_N),
\end{align}
so that NLP \eqref{eq:nlpdef} can be rewritten
\begin{align}
\label{eq:nlpnew}
&\minimise_z~J(z)\\
&\text{s.t.}~H(z)=0,~z\in\Zcal.\nonumber
\end{align}
\looseness-1
 \begin{ass}[Smoothness and semi-algebraicity]
 \label{ass:semi_alg}
The functions~${\big\{J_i\big\}_{i=1}^N}$, ${Q}$, ${\big\{F_i\big\}_{i=1}^N}$ and~${G}$ are twice continuously differentiable and semi-algebraic.
\end{ass}
\begin{rk}
The set of semi-algebraic functions is closed with respect to sum, product and composition. The indicator function of a semi-algebraic set is a semi-algebraic function.
\end{rk}
\looseness-1Semi-algebraicity is needed in order to apply the results of \cite{att2013}, which are valid for functions satisfying the Kurdyka-Lojasiewicz (KL) property, which is the case for all real semi-algebraic functions \cite{bolte2006}. From a control perspective, this means that our results are valid for polynomial systems subject to polynomial constraints and objectives.
\begin{ass}
\label{ass:sec_order}
\looseness-1The NLP \eqref{eq:nlpnew} admits an isolated KKT point ${((z^{\ast})^{\Trans},({\mu}^{\ast})^{\Trans},({\lambda}^{\ast})^{\Trans})^{\Trans}\in{\Rset}^{n+r+q}}$, which satisfies
\begin{itemize}
\item$z^{\ast}$ is regular,
\item${((z^{\ast})^{\Trans},({\mu}^{\ast})^{\Trans},({\lambda}^{\ast})^{\Trans})^{\Trans}}$ satisfies the \textit{second order optimality condition}, that is
\begin{align}
&p^{\Trans}{\nabla}^2L(z^{\ast},{\mu}^{\ast},{\lambda}^{\ast})p>0~\text{for~all}~p\in\Rset^n~\text{such~that}~\nonumber\\
&{\nabla}H(z^\ast)^{\Trans}p=0,~A_{I^{\ast}}p=0,
\end{align}
where~${I^{\ast}\subseteq\left\{1,\ldots,q\right\}}$ is the set of indices of active inequality constraints at ${z^{\ast}}$.
\item For all $i\in I^\ast$, $\lambda_i^\ast>0$. 
\end{itemize}
\end{ass}
\section{Outer loop: Partially augmented Lagrangian}
\label{sec:out_loop}
Given a penalty parameter ${\rho>0}$, the augmented Lagrangian associated with problem \eqref{eq:nlpnew} is defined as
\begin{align}
L_\rho(z,\mu) := J(z)+\mu^{\Trans}H(z)+\frac{\rho}{2}\left\|H(z)\right\|_2^2,
\end{align}
where ${\mu\in\Rset^r}$. Only equality constraints \({H(z)=0}\) are penalised. The polytopic inequalities are kept as constraints on the augmented Lagrangian. 
\subsection{Algorithm description}
\label{subsec:out_loop_dsc}
\looseness-1The outer loop is similar to a dual ascent on the multiplier estimates ${\mu_k}$, along with iterative updates of the penalty parameter ${\rho_k}$. At each outer iteration ${k}$, the primal problem    
\begin{align}
\label{eq:prim_pb}
\minimise_{z\in\Zcal}~L_{\rho_k}(z,\mu_k)
\end{align}
is solved to a given level of accuracy \({\epsilon_k>0}\). The outer loop is presented in Algorithm \ref{alg:mult_loop} below. The notation \({d(x,S)}\) stands for the distance function between  a point ${x}$ and a set ${S}$ in ${{\Rset}^n}$ and is defined by
\begin{align}
\label{eq:def_dis}
d(x,S):=\infim_{z\in{S}}\left\|x-z\right\|_2.
\end{align}
\begin{algorithm}[h!]
  \caption{\label{alg:mult_loop}Method of multipliers with partial constraint penalisation}
  \begin{algorithmic}
    \State \textbf{Input:} Objective \(J\), equality constraint \(H\), polytope \(\Zcal\), \({\beta>1}\) and final tolerance on nonlinear equality constraints \({\eta>0}\).
    	\begin{itemize}
    		\item Initial guess for optimiser ${z_0\in\Zcal}$, initial guess for multiplier estimates ${\mu_0\in\Rset^r}$, initial penalty parameter ${\rho_0>1}$, initial tolerance on optimality conditions ${\epsilon_0>0}$.
    	\end{itemize}
    \State \textbf{Initialization:}~\({z\leftarrow z_0}\), \({\mu\leftarrow\mu_0}\), \({\rho\leftarrow\rho_0}\), \({\epsilon\leftarrow\epsilon_0}\).
    \Repeat
    	\State Find \(\bar{z}\in\Zcal\) such that \({d(0,{\nabla}L_\rho(\bar{z},\mu)+\Ncal_\Zcal(\bar{z}))}\leq\epsilon\)\\~~~~using \(z\) as a warm-start
	\State ${z\leftarrow\bar{z}}$
	\State ${\mu\leftarrow\mu+{\rho}H(z)}$, ${\epsilon\leftarrow\epsilon\big/\rho}$, $\rho\leftarrow\beta\rho$
    \Until{${\left\|H(z)\right\|\leq\eta}$} 
   \end{algorithmic}
\end{algorithm}
\setlength{\textfloatsep}{5pt}
\looseness-1 At every outer iteration, a point ${\bar{z}}$ is found, which satisfies inexact optimality conditions. Then, the dual estimate ${\mu}$ is updated, the penalty parameter increased and the tolerance on the first order optimality conditions is reduced. The main tuning variables are the initial value of the penalty parameter $\rho_0$ and the growth coefficient ${\beta}$.
\subsection{Convergence analysis}
\label{subsec:out_loop_cv}
\looseness-1We prove that the augmented Lagrangian iterations are locally convergent to the isolated KKT point \({\big((z^\ast)^{\Trans},(\mu^\ast)^{\Trans},(\lambda^\ast)^{\Trans}\big)^{\Trans}}\) satisfying Assumption \ref{ass:sec_order}. Our analysis is based on the results of \cite{bert1982}. We start by showing that the inner problem is well-defined, which means that for any optimality tolerance \({\epsilon>0}\), there exists a point \({\bar{z}}\), as in Algorithm \ref{alg:mult_loop}, satisfying \({d(0,{\nabla}L_\rho(\bar{z},\mu)+\Ncal_\Zcal(\bar{z}))\leq\epsilon}\), under appropriate conditions on \({\rho}\) and \({\mu}\). 
\begin{lem}[Existence of an inner critical point]
There exists ${\delta>0}$, ${\bar\rho>0}$ and ${\kappa>0}$ such that
\begin{align}
\forall\big(\mu,\rho\big)\in\Scal:=\big\{\big(\mu,\rho\big)\in\Rset^{m+1}\big|\left\|\mu-\mu^\ast\right\|_2\leq\delta\rho,\rho\geq\bar{\rho}\big\},
\end{align}
the problem
\setlength{\textfloatsep}{5pt}
\begin{align}
\label{eq:inn_pb_ball}
&\minimise~L_\rho(z,\mu)\\
&\text{s.t.}~z\in\Zcal\nonumber\\
&~~~~\left\|z-z^\ast\right\|_2<\kappa\nonumber.
\end{align}
\end{lem}
has a unique minimiser. Moreover, for all \({\epsilon>0}\) and for all \({\big(\mu,\rho\big)\in\Scal}\), there exists \({z_\epsilon\in\Rset^n}\) such that
\begin{align}
d(0,{\nabla}L_\rho(z_\epsilon,\mu)+\Ncal_\Zcal(z_\epsilon))\leq\epsilon.
\end{align}
\begin{proof}
\looseness-1The first part of the statement is a direct consequence of Proposition \({2.11}\) in \cite{bert1982}, as the second order optimality condition is satisfied (Assumption \ref{ass:sec_order}). From Definition \ref{def:crit_pt}, taking \({\big(\mu,\rho\big)\in\Scal}\), this implies that we can find \({\tilde{z}\in\Zcal\cap\Bcal(z^\ast,\kappa)}\) such that \({d(0,{\nabla}L_\rho(\tilde{z},\mu)+\Ncal_{\Zcal\cap\Bcal(z^\ast,\kappa)}(\tilde{z}))\leq\epsilon}\) for all \({\epsilon>0}\), where \({\Bcal(z^\ast,\kappa)}\) is the open ball of radius \({\kappa}\) centered at \({z^\ast}\). As ${\ri\Bcal(z^\ast,\kappa)\cap\ri\Zcal\ne\emptyset}$, where ${\ri}$ stands for the relative interior, \({\Ncal_{\Zcal\cap\Bcal(z^\ast,\kappa)}(\tilde{z})=\Ncal_\Zcal(\tilde{z})\cap\Ncal_{\Bcal(z^\ast,\kappa)}(\tilde{z})}\), so that 
\begin{align}
d(0,{\nabla}L_\rho(\tilde{z},\mu)+\Ncal_\Zcal(\tilde{z}))&\leq{d}(0,{\nabla}L_\rho(\tilde{z},\mu)\nonumber\\
&~~~~+\Ncal_{\Zcal\cap\Bcal(z^\ast,\kappa)}(\tilde{z}))\leq\epsilon.
\end{align}
The end of the proof follows by taking ${z_\epsilon:=\tilde{z}}$.
\end{proof}
The next Lemma is a reformulation of the inexact optimality conditions \({d(0,{\nabla}L_{\rho}(\bar{z},{\mu})+{\Ncal}_{\Zcal}(\bar{z}))\leq\epsilon}\). 
\begin{lem}[Inexact optimality conditions]
~Let \({z\in\Rset^n}\), \({f\in C^1(\Rset^n,\Rset)}\) and \({\Zcal}\) be a convex set in \({\Rset^n}\). The following equivalence holds:
\begin{align}
d(0,{\nabla}f(z)+\Ncal_\Zcal(z))\leq\epsilon\Leftrightarrow\exists v\in\Rset^n~\text{such that}\nonumber\\
\left\{
\begin{aligned}
&\left\|v\right\|_2\leq\epsilon\\
&-{\nabla}f(z)\in\Ncal_\Zcal(z)+v.
\end{aligned}
\right.
\end{align}
\end{lem}
\begin{proof}
This directly follows from the definition of the distance as an infimum.
\end{proof}
As \({\big((z^\ast)^{\Trans},(\mu^\ast)^{\Trans},(\lambda^\ast)^{\Trans}\big)^{\Trans}}\) is an isolated KKT point, there exists \({\nu>0}\) such that \({z^\ast}\) is the only critical point of \eqref{eq:nlpnew} in \({\Bcal(z^{\ast},\nu)}\). 
From the statement of Algorithm \ref{alg:mult_loop}, the penalty parameter ${\rho}$ is guaranteed to increase to infinity and the optimality tolerance ${\epsilon}$ to converge to zero. The following convergence theorem is valid if the iterates ${z_k}$ stay in the ball ${\Bcal(z^{\ast},\min\left\{\nu,\kappa\right\})}$ for a large enough ${k}$. As noticed in \cite{bert1982}, it is generally the case in practice, as warm-starting the inner problem \eqref{eq:inn_pb_ball} on the previous solution leads the iterates to stay around the same local minimum $z^{\ast}$. Thus one can reasonably assume that ${z_k\in\Bcal(z^{\ast},\min\left\{\nu,\kappa\right\})}$.
\begin{theo}[Local convergence to a KKT point]
\label{th:out_cv}
Let ${\left\{z_k\right\}}$ and ${\left\{\mu_k\right\}}$ be sequences in ${\Rset^n}$ such that 
\begin{align} 
\left\{
\begin{aligned}
&z_k\in{\Zcal}\cap\cl\big(\Bcal\big(z^\ast,\min\left\{\kappa,\nu\right\}\big)\big)\\
&-{\nabla}L_{\rho_k}(z_k,\mu_k)\in\Ncal_\Zcal(z_k)+d_k
\end{aligned}
\right.
\end{align}
where 
\begin{itemize}
\item \({\left\{\rho_k\right\}}\) is increasing and \({\rho_k{\rightarrow}+\infty}\),
\item \({\left\{\mu_k\right\}}\) and \({\left\{\lambda_k\right\}}\), associated with $\Ncal_\Zcal\left(z_k\right)$, are bounded,
\item \({\left\|d_k\right\|_2\rightarrow0}\).
\end{itemize}
Assume that all limit points of ${\left\{z_k\right\}}$ are regular. We then have that \({z_k\rightarrow{z}^\ast}\) and \({\tilde{\mu}_k\rightarrow\mu^\ast}\), where \({\tilde{\mu}_k:=\mu_k+{\rho_k}H(z_k)}\), ${z^{\ast}}$ and ${\mu^{\ast}}$ are defined in Assumption \ref{ass:sec_order}.
\end{theo}
\begin{proof}
As ${\Zcal\cap\cl\big(\Bcal\big(z^\ast,\min\left\{\kappa,\nu\right\}\big)\big)}$ is compact, by Weierstrass theorem, there exists an increasing mapping ${\phi:\Nset\rightarrow\Nset}$ and a point ${z'\in\Zcal\cap\cl\big(\Bcal\big(z^\ast,\min\left\{\kappa,\nu\right\}\big)\big)}$ such that \({\left\{z_{\phi(k)}\right\}}\) converges to \({z'}\), which is regular by assumption. We also know by assumption that 
\begin{align}
\exists v_{\phi(k)}&\in\Ncal_{\Zcal}(z_{\phi(k)}),\nonumber\\
0&={\nabla}L_{\rho_{\phi(k)}}(z_{\phi(k)},\mu_{\phi(k)})+v_{\phi(k)}+d_{\phi(k)}.
\end{align}
By definition of the normal cone to ${\Zcal}$ at ${z_{\phi(k)}}$, this means that 
\begin{align}
\exists \lambda_{\phi(k)}\geq0,\left\{\begin{aligned}
&{\nabla}L_{\rho_{\phi(k)}}(z_{\phi(k)},\mu_{\phi(k)})+A^{\Trans}\lambda_{\phi(k)}+d_{\phi(k)}=0\\
&\lambda_{\phi(k)}\geq0,~{\lambda}_{\phi(k)}^{\Trans}(b-Az_{\phi(k)})=0.
\end{aligned}
\right.
\end{align}
We thus have
\begin{align}
\label{eq:stat_inex_k}
{\nabla}J(z_{\phi(k)})+{\nabla}H(z_{\phi(k)})^{\Trans}\tilde{\mu}_{\phi(k)}+A^{\Trans}{\lambda}_{\phi(k)}=-d_{\phi(k)}.
\end{align}
As \({z'}\) is regular, there exists \({K\in\Nset_+}\) such that \({{\nabla}H(z_{\phi(k)})}\) is full-rank for all \({k\geq K}\). Subsequently, for all ${k\geq K}$,
\begin{align}
\tilde{\mu}_{\phi(k)}=&\big({\nabla}H(z_{\phi(k)}){\nabla}H(z_{\phi(k)})^{\Trans}\big)^{-1}{\nabla}H(z_{\phi(k)})\big(-d_{\phi(k)}\nonumber\\
&-{\nabla}J(z_{\phi(k)})-A^{\Trans}\lambda_{\phi(k)}\big).
\end{align}
As \({z_{\phi(k)}\rightarrow z'}\), by continuity of \({\lambda_{\phi(k)}^{\Trans}\big(b-Az_{\phi(k)}\big)=0}\), there exists \({\lambda'\geq0}\) such that 
\begin{align}
\lambda_{\phi(k)}\rightarrow\lambda',~\big(\lambda'\big)^{\Trans}(b-Az')=0.
\end{align}
As \({d_{\phi(k)}\rightarrow0}\), by continuity of \({{\nabla}H}\), this implies that \({\tilde{\mu}_{\phi(k)}{\rightarrow}\tilde{\mu}'}\), where
\begin{align}
\tilde{\mu}':=\big({\nabla}H(z'){\nabla}H(z')^{\Trans}\big)^{-1}{\nabla}H(z')\big(-{\nabla}J(z')-A^{\Trans}\lambda'\big).
\end{align}
By taking limit in \eqref{eq:stat_inex_k}, we obtain
\begin{align}
\left\{
\begin{aligned}
&{\nabla}J(z')+{\nabla}H(z')^{\Trans}{\tilde{\mu}}'+A^{\Trans}{\lambda}'=0\\
&\lambda'{\geq}0,~(\lambda')^{\Trans}(b-Az')=0.
\end{aligned}
\right.
\end{align}
Feasibility of ${z'}$ immediately follows from the convergence of $\tilde{\mu}_{\phi(k)}$ and the fact that $\rho_{\phi(k)}\!\rightarrow\!+\infty$. Thus we have $H(z')\!=\!0$.
This implies that $\big(\big(z'\big)^{\Trans},\big(\big(\tilde{\mu}'\big)^{\Trans},\big(\big(\lambda'\big)^{\Trans}\big)^{\Trans}$ satisfies the KKT conditions.\\
As \(z'\!\in\!\Zcal\cap\cl\Bcal\big(z^\ast,\min\left\{\kappa,\nu\right\}\big)\), in which \(z^{\ast}\) is the unique KKT point by assumption, one can claim that \(z'\!=\!z^{\ast}\). Taking the KKT conditions on \(z^{\ast}\) and \(z'\), it follows that 
\begin{align}
{\nabla}H(z^\ast)^{\Trans}{\tilde{\mu}}^{\ast}+A_{I^\ast}^{\Trans}{\lambda}_{I^\ast}^{\ast}={\nabla}H(z')^{\Trans}{\tilde{\mu}}'+A_{I'}^{\Trans}{\lambda'}_{I'},
\end{align}
where \({I^\ast}\) and \({I'}\) are the sets of indices of active constraints at \({z^\ast}\) and \({z'}\) respectively. Obviously, as \({z^\ast=z'}\), \({I^\ast=I'}\). Thus
\begin{align} 
{\nabla}H(z^\ast)^{\Trans}(\tilde{\mu}^\ast-\tilde{\mu}')+A_{I^\ast}^{\Trans}(\lambda_{I^\ast}-\lambda_{I'})=0.
\end{align}
As \({z^\ast}\) is regular, we finally obtain that \({\tilde{\mu}'=\tilde{\mu}^\ast}\). As all limit points of ${\left\{z_k\right\}}$ converge to ${z^{\ast}}$, one can conclude that ${z_k\rightarrow z^\ast}$ and ${\tilde{\mu}_k\rightarrow \mu^{\ast}}$.
\end{proof}
\begin{rk}
\label{rk:cv_loc}
The convergence result of Theorem \ref{th:out_cv} is local. Yet convergence can be globalised (meaning that the assumption according to which the iterates lie in a ball centered at ${z^\ast}$ can be removed) by applying the dual update scheme proposed in \cite{con1991}.
\end{rk}
\section{Inner loop: Inexact proximal regularised BCD}
\label{sec:inn_loop}
\looseness-1The non-convex inner problem \eqref{eq:prim_pb} is solved in a distributed manner. An inexact BCD scheme is proposed, which is based on the abstract result of \cite{att2013}. The main idea is to perform local convex approximations of each agent's cost and compute descent updates from it. Contrary to the approach of \cite{neco2009}, where splitting is, in some sense, applied after convexification, we show that convexification can be performed after splitting, under the assumption that the KL property is satisfied.
\subsection{Algorithm description}
\label{subsec:inn_loop_dsc}
The proposed algorithm is a proximal regularised inexact BCD, which is based on Theorem \({6.2}\) in \cite{att2013}, providing general properties ensuring global convergence to a critical point of the nonsmooth objective, assuming that it satisfied the KL property. The inner minimisation method is presented in Algorithm \ref{algo:inn_loop} below.
The positive real numbers \({\left\{\alpha_i^l\right\}_{i=1}^N}\) are the regularisation parameters. We assume that 
\begin{align}
\label{eq:def_reg_bnds}
\forall i\in\left\{1,{\ldots},N\right\},\exists \alpha_i^-,\alpha_i^+>0~\text{such that}\nonumber\\
\left\{
\begin{aligned}
&{\alpha}_i^-<{\alpha}_i^+\\
&{\forall}l\in\Nset,\alpha_i^l\in\left[{\alpha}_i^-,{\alpha}_i^+\right].
\end{aligned}
\right.
\end{align}
\begin{algorithm}
	\caption{\label{algo:inn_loop}Inexact proximal regularised BCD}
	\begin{algorithmic}
		\State \textbf{Input:}~Augmented Lagrangian~\({L_{\rho}(.,{\mu})}\),~polytopes \({\left\{\Zcal_i\right\}_{i=1}^N}\) and termination tolerance~\({\tau>0}\).
		\State \textbf{Initialization:} ${z_1^0\in\Zcal_1,\ldots,z_N^0\in\Zcal_N}$\par
							  ${l\leftarrow0}$.
		\While{${\left\|z^{l+1}-z^l\right\|_{\infty}>\tau}$}
			\For{${i\in\left\{1,\ldots,N\right\}}$}
			\State\hspace{-0.45cm}\({z_{i}^{l+1}\leftarrow{\argm}_{z_i\in\Zcal_i}{\nabla}_{z_i}L_{\rho}(z_1^{l+1},{\ldots},z_N^l,{\mu})^{\Trans}(z_i-z_i^l)}\)\par
				  ~~~~~~~~~~~~~~\({+\frac{1}{2}(z_i-z_i^l)^{\Trans}(B_i^l+{\alpha}_i^lI_{n_i})(z_i-z_i^l)}\)
			\EndFor  
			\State \({l\leftarrow l+1}\)
		\EndWhile
	\end{algorithmic}
\end{algorithm}
Matrices ${B_i^l}$ are positive definite and need to be chosen carefully in order to guarantee convergence of Algorithm \ref{algo:inn_loop}, as explained in paragraph \ref{subsec:inn_loop_cv}. Thus, Algorithm \ref{algo:inn_loop} consists in solving convex QPs sequentially. Under appropriate assumptions \cite{bert1997}, computations can be partly parallelised. Moreover, computational efficacy can be improved by warm-starting.  
\subsection{Convergence analysis}
\label{subsec:inn_loop_cv}
When considering the minimisation of functions \({f:\Rset^{n_1}\times\ldots\times\Rset^{n_N}\rightarrow\Rset\cup\left\{+\infty\right\}}\) satisfying the KL property \cite{bolte2006} and having the following structure
\begin{align}
\label{eq:sep_struc}
f(z)=\sum_{i=1}^Nf_i(z_i)+P(z_1,{\ldots},z_N),
\end{align}
where \({f_i}\) are proper lower semicontinuous and \({P}\) is twice continuously differentiable, two key assumptions need to be satisfied by the iterates \({z_i^l}\) of the inexact BCD scheme in order to ensure global convergence \cite{att2013}: a \textit{sufficient decrease} property and the \textit{relative error} condition.
The \textit{sufficient decrease} property asserts that 
\begin{align}
\label{eq:suff_dec}
&\forall i\in\left\{1,\ldots,N\right\},\forall l\in\Nset,\nonumber\\
&f_i(z_i^{l+1})+P(z_1^{l+1},\ldots,z_i^{l+1},\ldots,z_N^l)+{\frac{\alpha_i^l}{2}}\left\|z_i^{l+1}-z_i^l\right\|_2^2\nonumber\\
&~~~~~~~~~~~~~~~~\leq f_i(z_i^l)+P(z_1^{l+1},\ldots,z_{i-1}^{l+1},z_i^l,\ldots,z_N^l)\enspace.
\end{align}
The \textit{relative error} condition states that 
\begin{align}
\label{eq:rel_err}
&\forall i\in\big\{1,\ldots,N\big\},\exists b_i>0~\text{such that}~{\forall}l\in\Nset,\nonumber\\
&\exists v_i^{l+1}\in\partial f_i(z_i^{l+1}),\nonumber\\
&\left\|v_i^{l+1}+{\nabla}_{z_i}P(z_1^{l+1},\ldots,z_{i}^{l+1},z_{i+1}^l,\ldots,z_N^l)\right\|_2\nonumber\\
&~~~~~~~~~~~~~~~~~~~~~~~~~~~~~~~~~~~~~~~~~~~~\leq b_i\left\|z_i^{l+1}-z_i^l\right\|_2\enspace.
\end{align}
This means that at every iteration one can find a vector in the sub-differential of ${f_i}$ computed at the current iterate, which is bounded by the momentum of the sequence ${z_i^{l+1}-z_i^l}$. It is actually a key step in order to apply the KL inequality, as done in \cite{att2013}. For clarity, we state Theorem ${6.2}$ in \cite{att2013}.
\begin{theo}[Convergence of inexact BCD]
\label{th:att_2013}
~\looseness-1Let \({f}\) be a proper lower semi-continuous function with structure \eqref{eq:sep_struc} satisfying the KL property and bounded from below. Let ${\underline{\alpha}>0}$ such that ${\underline{\alpha}\leq\alpha_i^l}$ for all $i\in\left\{1,\ldots,N\right\}$ and $l\in\Nset$. Let \({\left\{z^l\right\}}\) be a sequence satisfying \eqref{eq:suff_dec} and \eqref{eq:rel_err}. If \({\left\{z^l\right\}}\) is bounded, then \({\left\{z^l\right\}}\) converges to a critical point of \({f}\). 
\end{theo}
\begin{rk}
Note that \({{\delta}_{\Zcal}+L_{\rho}(.,{\mu})}\) has the structure \eqref{eq:sep_struc}.
\end{rk}
\looseness-1Our convergence analysis then mainly consists in verifying that the iterates generated by Algorithm \ref{algo:inn_loop} satisfy \eqref{eq:suff_dec} and \eqref{eq:rel_err}. 
The following definitions are necessary for the remainder of the proof:
\begin{align}
\label{eq:use_defs}
&\forall i\in\left\{1,\ldots,N\right\},\forall l\in\Nset,\nonumber\\
&S(z_1^{l+1},\ldots,z_{i-1}^{l+1},z_i,z_{i+1}^l,{\ldots},z_N^l):=\nonumber\\
&~~~~~~~~~~~~~~~~~~{\mu}_G^{\Trans}G(z_1^{l+1},{\ldots},z_{i-1}^{l+1},z_i,z_{i+1}^l,{\ldots},z_N^l)\nonumber\\
&~~~~~~~~~~~~~~~~~~+{\frac{\rho}{2}}\left\|G(z_1^{l+1},{\ldots},z_{i-1}^{l+1},z_i,z_{i+1}^l,{\ldots},z_N^l)\right\|_2^2\enspace,\nonumber\\
&R_i^l(z_i):=Q(z_1^{l+1},{\ldots},z_{i-1}^{l+1},z_i,z_{i+1}^l,{\ldots},z_N^l)\nonumber\\
&~~~~~~~~~~+S(z_1^{l+1},{\ldots},z_{i-1}^{l+1},z_i,z_{i+1}^l,{\ldots},z_N^l)\enspace,\nonumber\\
&\Psi_i(z_i):={\mu}_i^{\Trans}F_i(z_i)+{\frac{\rho}{2}}\left\| F_i(z_i)\right\|_2^2\enspace,\nonumber\\
&L_i^l(z_i):=J_i(z_i)+{\Psi}_i(z_i)+R_i^l(z_i)\enspace,\nonumber\\
&\Gamma_i(z_i){:=}J_i(z_i)+\Psi_i(z_i)+{\delta}_{\Zcal_i}(z_i)\enspace,\nonumber\\
&\Phi_i^l(z_i){:=}L_i^l(z_i)+{\delta}_{\Zcal_i}(z_i)\enspace,
\end{align}
where ${{\mu}_G\in\Rset^r}$ is the subvector of multipliers associated with the equality constraint ${G(z_1,{\ldots},z_N)=0}$. 
\begin{lem}[Bound on hessian norm]
\begin{align}
\forall i\in\left\{1,\ldots,N\right\},\forall l\in\Nset,~\left\|{\nabla}^2L_i^l\right\|_\infty^{{\Zcal}_i}\leq C_i,
\end{align}
where
\begin{align}
C_i:=\left\|{\nabla}^2J_i\right\|_{\infty}^{\Zcal_i}+\left\|{\nabla}^2\Psi_i\right\|_{\infty}^{\Zcal_i}+\max_{i\in\left\{1,\ldots,N\right\}}\left\|{\nabla}_i^2(S+Q)\right\|_\infty^{\Zcal}.
\end{align}
\end{lem}
\begin{proof}
The proof directly follows from the definitions \eqref{eq:use_defs} and the fact that the functions involved in \eqref{eq:use_defs} are twice continuously differentiable over compact sets ${\Zcal_i}$.
\end{proof}
\begin{ass}
\label{ass:hss_approx}
The matrices ${B_i^l}$ can be chosen so that ${B_i^l-C_iI_{n_i}\succ0}$ and ${2C_iI_{n_i}-B_i^l\succ0}$, for all ${i\in\left\{1,\ldots,N\right\}}$ and all ${l\in\Nset}$.
\end{ass}
\begin{rk}
Such an assumption requires knowledge of the Lipschitz constant of the gradient of the augmented Lagrangian. However, a simple backtracking procedure may be applied in practice. 
\end{rk}
\begin{lem}[Sufficient decrease]
For all ${i\in\left\{1,\ldots,N\right\}}$ and all ${l\in\Nset}$,
\begin{align}
{\Phi}_i^l(z_i^{l+1})+{\frac{\alpha_i^l}{2}}\left\|z_i^{l+1}-z_i^l\right\|_2^2\leq{\Phi}_i^l(z_i^l).
\end{align}
\end{lem}
\begin{proof}
By definition of ${z_i^{l+1}}$ in Algorithm \ref{algo:inn_loop}, 
\begin{align}
&\forall z_i\in\Zcal_i,\nonumber\\
&\nabla L_i^l(z_i^l)^{\Trans}(z_i^{l+1}-z_i^l)\nonumber\\
&~~~~~+{\frac{1}{2}}(z_i^{l+1}-z_i^l)^{\Trans}(B_i^l+{\alpha}_i^lI_{n_i})(z_i^{l+1}-z_i^l)\nonumber\\
&\leq {\nabla}L_i^l(z_i^l)^{\Trans}(z_i-z_i^l)+{\frac{1}{2}}(z_i-z_i^l)^{\Trans}(B_i^l+{\alpha}_i^lI_{n_i})(z_i-z_i^l), 
\end{align}
which by substituting ${z_i}$ with ${z_i^l\in\Zcal_i}$ implies that
\begin{align}
\label{eq:ineq_z_i_l}
&\nabla L_i^l(z_i^l)^{\Trans}(z_i^{l+1}-z_i^l)\nonumber\\
&~~~~~~~+{\frac{1}{2}}(z_i^{l+1}-z_i^l)^{\Trans}(B_i^l+{\alpha}_i^lI_{n_i})(z_i^{l+1}-z_i^l)\leq 0.
\end{align}
However, by applying the descent Lemma, one gets
\begin{align}
L_i^l(z_i^{l+1})\leq L_i^l(z_i^l)+{\nabla}L_i^l(z_i^l)^{\Trans}(z_i^{l+1}-z_i^l)+{\frac{C_i}{2}}\left\|z_i^{l+1}-z_i^l\right\|_2^2.
\end{align}
Combining this last inequality with inequality \eqref{eq:ineq_z_i_l}, one obtains
\begin{align}
L_i^l(z_i^{l+1})&\leq L_i^l(z_i^l)\nonumber\\
&+{\frac{1}{2}}(z_i^{l+1}-z_i^l)^{\Trans}((C_i-{\alpha_i^l})I_{n_i}-B_i^l)(z_i^{l+1}-z_i^l),
\end{align}
which, by Assumption \ref{ass:hss_approx}, implies that
\begin{align}
L_i^l(z_i^{l+1})+{\frac{\alpha_i^l}{2}}\left\|z_i^{l+1}-z_i^l\right\|_2^2 \leq L_i^l(z_i^l).
\end{align}
However, ${{\delta}_{\Zcal_i}(z_i^{l+1})={\delta}_{\Zcal_i}(z_i^l)=0}$. Thus,
\begin{align}
\Phi_i^l(z_i^{l+1})+{\frac{\alpha_i^l}{2}}\left\|z_i^{l+1}-z_i^l\right\|_2^2\leq\Phi_i^l(z_i^l).
\end{align}
\end{proof}
\begin{lem}[Relative error condition]
${{\forall}i\in\left\{1,\ldots,N\right\}}$,~${{\forall}l\in\Nset}$,
\begin{align}
&\exists v_i^{l+1}\in{\partial}{\Gamma}_i(z_i^{l+1}),\nonumber\\
&\left\|v_i^{l+1}+{\nabla}(Q+S)(z_1^{l+1},\ldots,z_i^{l+1},z_{i+1}^{l},{\ldots},z_N^l)\right\|_2\nonumber\\
&~~~~~~~~~~\leq b_i\left\|z_i^{l+1}-z_i^l\right\|_2,
\end{align}
where ${b_i:=3C_i+{\alpha}_i^+}$ and ${{\alpha}_i^+}$ is defined in \eqref{eq:def_reg_bnds}.
\end{lem}
\begin{proof}
Writing the necessary optimality conditions for the ${i^{\text{th}}}$ QP of Algorithm \ref{algo:inn_loop}, one obtains
\begin{align}
0\in{\nabla}L_i^l(z_i^l)+\big(B_i^l+\alpha_i^lI_{n_i}\big)(z_i^{l+1}-z_i^l)+\Ncal_{\Zcal_i}(z_i^{l+1}).
\end{align}
Thus there exists ${w_i^{l+1}\in{\Ncal}_{\Zcal_i}(z_i^{l+1})}$ such that 
\begin{align}
0&=w_i^{l+1}+{\nabla}L_i^l(z_i^l)+(B_i^l+{\alpha}_i^lI_{n_i})(z_i^{l+1}-z_i^l)\nonumber\\
&=w_i^{l+1}+{\nabla}L_i^l(z_i^{l+1})+{\nabla}L_i^l(z_i^l)-{\nabla}L_i^l(z_i^{l+1})\nonumber\\
&~~~+(B_i^l+{\alpha}_i^lI_{n_i})(z_i^{l+1}-z_i^l).
\end{align}
The last equality yields
\begin{align}
w_i^{l+1}&+{\nabla}J_i(z_i^{l+1})+{\nabla}F_i(z_i^{l+1})^{\Trans}{\mu}\nonumber\\
&+{\rho}{\nabla}F_i(z_i^{l+1})^{\Trans}F_i(z_i^{l+1})\nonumber\\
&+{\nabla}_i(Q+S)(z_1^{l+1},{\ldots},z_{i}^{l+1},z_{i+1}^l,{\ldots},z_N^l)\nonumber\\
&={\nabla}L_i^l(z_i^{l+1})-{\nabla}L_i^l(z_i^l)+(B_i^l+{\alpha}_i^lI_{n_i})(z_i^l-z_i^{l+1}),
\end{align}
from which it immediately follows
\begin{align}
&\exists v_i^{l+1}\in{\partial}{\Gamma}_i(z_i^{l+1}),\nonumber\\
&v_i^{l+1}+{\nabla}_i(Q+S)(z_1^{l+1},{\ldots},z_{i-1}^{l+1},z_i^{l+1},z_{i+1}^l,{\ldots},z_N^l)\nonumber\\
&={\nabla}L_i^l(z_i^{l+1})-{\nabla}L_i^l(z_i^l)+(B_i^l+{\alpha}_i^lI_{n_i})(z_i^l-z_i^{l+1}).
\end{align}
However, as ${L_i^l\in C^2({\Zcal}_i,\Rset)}$, one gets
\begin{align}
&\left\|v_i^{l+1}+{\nabla}_i(Q+S)(z_1^{l+1},{\ldots},z_i^{l+1},z_{i+1}^l,\ldots,z_N^l)\right\|_2\nonumber\\
&~~~~\leq(C_i+\left\|B_i^l\right\|_2+\alpha_i^l)\left\|z_i^{l+1}-z_i^l\right\|_2\nonumber\\
&~~~~\leq(3C_i+{\alpha}_i^+)\left\|z_i^{l+1}-z_i^l\right\|_2.
\end{align}
\end{proof}
Now that the two main ingredients are proven, one can state the theorem guaranteeing convergence of the iterates of Algorithm \ref{algo:inn_loop} to a critical point of ${{\delta}_{\Zcal}+L_{\rho}(.,{\mu})}$.
\begin{theo}[Convergence of Algorithm \ref{algo:inn_loop}]
The sequence ${\big\{z^l\big\}}$ generated by Algorithm \ref{algo:inn_loop} converges to a point ${z^\ast}$ satisfying
\begin{align}
0\in{\nabla}L_{\rho}(z^\ast,{\mu})+{\Ncal}_{\Zcal}(z^\ast)
\end{align}
or equivalently ${d(0,{\nabla}L_{\rho}(z^\ast,{\mu})+{\Ncal}_{\Zcal}(z^\ast))=0}$.
\end{theo}
\begin{proof}
The sequence ${\left\{z^l\right\}}$ is obviously bounded, as the iterates are constrained to stay in the polytope ${\Zcal}$. The function \({{\delta}_{\Zcal}+L_{\rho}(.,{\mu})}\) is bounded from below. Moreover, by Assumption \ref{ass:semi_alg}, \({{\delta}_{\Zcal}+L_{\rho}(.,{\mu})}\) satisfies the KL property for all ${\mu\in\Rset^m}$ and ${\rho>0}$. The sequence ${\left\{z^l\right\}}$ satisfies the conditions \eqref{eq:suff_dec} and \eqref{eq:rel_err}. The parameter $\underline{\alpha}$ of Theorem \ref{th:att_2013} can be taken as $\min\left\{\alpha_i^-\big|i\in\left\{1,\ldots,N\right\}\right\}$. Convergence to a critical point ${z^\ast}$ of ${\delta_{\Zcal}+L_{\rho}(.,\mu)}$ then follows by applying Theorem \ref{th:att_2013}.
\end{proof}
We can then guarantee that, given ${\epsilon>0}$, by taking a sufficiently large number of inner iterations, Algorithm \ref{algo:inn_loop} converges to a point ${\bar{z}\in\Zcal}$ satisfying ${d(0,{\nabla}L_{\rho}(\bar{z},\mu)+{\Ncal}_{\Zcal}(\bar{z}))\leq\epsilon}$, as needed by Algorithm \ref{alg:mult_loop}.
\section{Numerical example}
\label{sec:num_ex}
Our algorithm is tested on the following class of non-convex NLPs:
\begin{align}
\label{eq:toy_pb_1}
&\minimise_{x_1,\ldots,x_N}\sum_{i=1}^{N}x_i^{\Trans}H_ix_i+\sum_{i=1}^{N-1}x_i^{\Trans}H_{i,i+1}x_{i+1}\nonumber\\
&\text{s.t.}\nonumber\\
&\|x_i\|_2^2=a^2,~i\in\big\{1,{\ldots},N\big\}\nonumber\\
&-b\leq{}x_{i,j}\leq{}b, i\in\big\{1,\ldots,N\big\}, j\in\big\{1,\ldots,d\big\},
\end{align}
where ${N}$ is the number of agents and ${d}$ their dimension. Matrices ${H_i}$ and ${H_{i,i+1}}$ are indefinite. When applying Algorithm \ref{algo:inn_loop} to minimise the augmented Lagrangian associated with \eqref{eq:toy_pb_1}, the update of agent ${i}$ only depends on agents ${i{-}1}$ and ${i{+}1}$, which implies that, after appropriate re-ordering \cite{bert1997}, every inner iteration actually consists of two sequential update steps where ${N/{\displaystyle}2}$ QPs are solved in parallel. We fix ${d=3}$, ${N=20}$, ${a=\sqrt{R}}$ and ${b=0.6R}$ with ${R=2}$. Matrices ${H_i}$ and ${H_{i,i+1}}$ are randomly generated. For a fixed number of outer iterations, we gradually increase the number of inner iterations and count the number of random problems, on which the algorithm ends up within a fixed feasibility tolerance with respect to the nonlinear equality constraints ${\left\|x_i\right\|_2^2=a^2}$. Statistics are reported for ${500}$ random NLPs of the type \eqref{eq:toy_pb_1}. The initial penalty parameter is set to ${\rho_0=0.1}$ and the growth coefficient ${\beta=100}$. The hessian matrices of Algorithm \ref{algo:inn_loop} are set to ${30\rho I_d}$. The algorithm is initialised on random primal and dual initial guesses, feasible with respect to the inequality constraints. Results for a feasibility tolerances of ${10^{-3}}$, ${10^{-4}}$ and ${10^{-6}}$ are presented in Figure \ref{fig:prec_quant_3_6}.
\begin{figure}[t]
\begin{picture}(100,140)
\put(0,0){\includegraphics[width=0.5\textwidth]{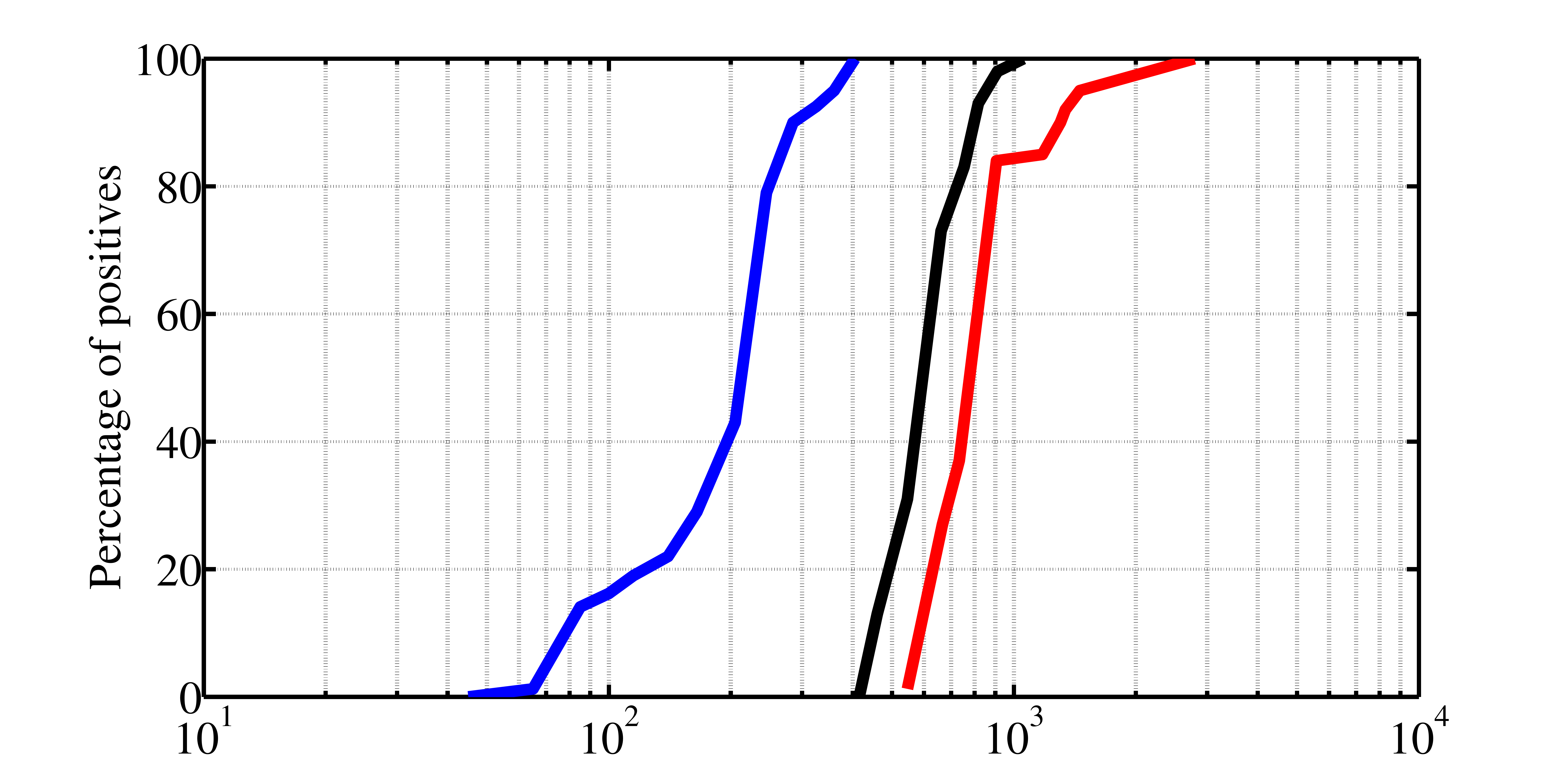}}
\put(90,-1){\footnotesize Total number of iterations}
\end{picture}
\caption{\label{fig:prec_quant_3_6}Percentage of positive problems for a fixed feasibility tolerance on the nonlinear equality constraints. Feasibility margins ${10^{-3}}$ in blue, ${10^{-4}}$ in black and ${10^{-6}}$ in red.}
\end{figure}
It clearly appears that obtaining an accurate feasibility margin requires a large number of iterations, yet a reasonable feasibility ($10^{-3}$) can be obtained with approximately hundred iterations. 
\bibliographystyle{plain}	
\bibliography{biblio}	

%
%
%
%
%
%
%
%

\end{document}